\documentclass{amsart}

\usepackage{amsfonts,amsmath,amssymb,latexsym,xspace,amsthm}
\newcommand{\C}{\mathbb C}
\newcommand{\Z}{\mathbb Z}
\newcommand{\N}{\mathbb N}
\newcommand{\Y}{\mathcal Y}
\newcommand{\D}{\mathrel{\mathcal D}}
\newcommand{\T}{\mathcal T}
\newcommand{\Cplx}{\mathcal C}

\newcommand{\ZnZ}{\mathbb Z_n}

\newcommand{\ld}{\left\lfloor}
\newcommand{\rd}{\right\rfloor}
\newcommand{\s}{\ld s \rd}
\newcommand{\up}[1]{\textup{#1}}
\newcommand{\Gwr}[1]{G \wr #1}
\newcommand{\dom}{\textup{dom}}
\newcommand{\ran}{\textup{ran}}
\newcommand{\rk}{\textup{rk}}

\newcommand{\oneton}{\{1,\ldots,n\}}
\newcommand{\onetok}{\{1,\ldots,k\}}
\newcommand{\perm}{\textup{perm}}

\newcommand{\IRR}{\textup{IRR}}

\theoremstyle{plain} \newtheorem{cor}{Corollary}[section]
\theoremstyle{plain} 
\theoremstyle{plain} 
\theoremstyle{plain} \newtheorem{thm}[cor]{Theorem}
\theoremstyle{definition} \newtheorem{defn}[cor]{Definition}
\theoremstyle{plain} \newtheorem*{thma}{Theorem}
\theoremstyle{plain}  
\theoremstyle{definition} 
\theoremstyle{plain}  
\theoremstyle{plain}  
\theoremstyle{definition} 
\theoremstyle{definition} 
\theoremstyle{definition}

\begin{document}

\begin{abstract}
We extend the theory of fast Fourier transforms on finite groups to finite inverse semigroups. We use a general method for constructing the irreducible representations of a finite inverse semigroup to reduce the problem of computing its Fourier transform to the problems of computing Fourier transforms on its maximal subgroups and a fast zeta transform on its poset structure. We then exhibit explicit fast algorithms for 
particular inverse semigroups of interest---specifically, for the rook monoid and its wreath products by arbitrary finite groups.
\end{abstract}




\title{Fast Fourier Transforms for Finite Inverse Semigroups}
\author{Martin E. Malandro}
\address{Department of Mathematics and Statistics, Box 2206, Sam Houston State University, Huntsville, TX 77341-2206}
\email{malandro@shsu.edu}




\thanks{The author was partially supported by AFOSR under grant FA9550-06-1-0027.}


\maketitle

\section{Introduction}
\label{SecIntro}

Given a complex-valued function $f$ on a finite group $G$, we may view $f$ as an element of the group algebra $\C G$ by identifying the natural basis of $\C G$ with the characteristic functions of the elements $g\in G$. That is,
\[
f=\sum_{g\in G}f(g)\delta_g
\]
corresponds to
\[
\sum_{g\in G}f(g)g \in \C G.
\]
Because $\C G$ is a semisimple algebra, it is the direct sum of its minimal (two-sided) ideals $M_i$:
\[
\C G = M_1 \oplus \cdots \oplus M_n.
\]
By taking a basis for each of the $M_i$ subject to a ``normalization" condition explained in Section \ref{SecFourierTransforms}, we obtain a basis for $\C G$ known as a {\em Fourier basis}. The {\em Fourier transform} \index{Fourier transform} of a function $f$ is then its re-expression in terms of a Fourier basis. 

As an example, let $G=\ZnZ$, the cyclic group of order $n$. An element $f$ of the group algebra $\C \ZnZ$ expressed with respect to the natural basis may be viewed as a signal, sampled at $n$ evenly spaced points in time. In this case, the minimal ideals of $\C \ZnZ$ are all 1-dimensional, so a Fourier basis must be unique up to scaling factors, and the Fourier basis here is indeed the usual basis of exponential functions given by the classical discrete Fourier transform. The re-expression of $f$ in terms of a Fourier basis thus corresponds to a re-expression of $f$ in terms of the frequencies that comprise $f$. This change of basis may be computed efficiently with the help of the classical fast discrete Fourier transform (FFT). 

A naive computation the Fourier transform of $f \in \C G$ requires $|G|^2$ operations. An {\em operation} \index{operation} is defined to be a complex multiplication followed by a complex addition. The problem of efficiently computing the Fourier transform of an arbitrary $\C$-valued function on $G$ has been considered for a wide range of groups $G$, and efficient algorithms for computing this change of basis now exist for many finite groups. For a survey of these results see, e.g., \cite{Dan1990,DanDiameters,DanSchurs,DanSepVars,DanAMS,Dan2004}. For example, it is known that the Fourier transform of $f \in \C G$ requires no more than:
\begin{itemize}
\item $O(n \log n)$ operations if $G=\ZnZ$ \cite{Bluestein,CooleyTukey},
\item $O(|S_n| \log^2|S_n|)$ operations if $G=S_n$, the symmetric group on $n$ elements \cite{Maslen}, and
\item $O(|B_n| \log^4|B_n|)$ operations if $G=B_n$, the hyperoctahedral group (that is, the signed symmetric group) on $n$ elements \cite{DanWreath}.
\end{itemize}
We shall define a {\em fast Fourier transform} \index{fast Fourier transform (FFT)}(FFT) for (or on) a finite group $G$ to be a procedure for calculating the Fourier transform of an arbitrary complex-valued function on $G$ which compares favorably to the naive algorithm. In general, $O(|G|\log^c|G|)$ algorithms are the goal in group FFT theory, although there exist families of groups $G$ for which there exist greatly improved---yet not $O(|G|\log^c|G|)$---algorithms, such as the family of matrix groups over a finite field \cite{DanDiameters}.


In \cite{RookFFT} we extended the theory of finite-group FFTs to create FFTs for a particular inverse semigroup known as the rook monoid. In this paper we extend the theory of finite-group FFTs to all finite inverse semigroups. In particular, we provide a method for building FFTs on arbitrary finite inverse semigroups and we construct $O(|S|\log^c|S|)$ FFTs for specific inverse semigroups $S$ of interest. Our main results are these.

\begin{thma}[Theorem \ref{MyBigThm}]Let $S$ be a finite inverse semigroup with $\D$-classes $D_0,\ldots,D_n$. Let $r_k$ denote the number of idempotents in $D_k$. Choose an idempotent $e_k$ from each $\D$-class $D_k$, and let $G_k$ be the maximal subgroup of $S$ at $e_k$. Then the number of operations required to compute the Fourier transform of an arbitrary $\C$-valued function $f$ on $S$ is no more than
\[
\Cplx(\zeta_S) + \sum_{k=0}^n r_k^2 \Cplx(G_k),
\]
where $\Cplx(\zeta_S)$ is the maximum number of operations needed to compute the zeta transform of $f$ on $S$ and $\Cplx(G_k)$ is the maximum number of operations needed to compute the Fourier transform of an arbitrary $\C$-valued function on $G_k$.
\end{thma}

\begin{thma}[Theorem 
\ref{ThmCplxSemigroupRn}]If $S=R_n$, the rook monoid on $n$ elements, then the Fourier transform of an arbitrary $\C$-valued function on $S$ may be computed in $O(|S|\log^3|S|)$ operations.
\end{thma}

\begin{thma}[Theorem 
\ref{ThmCplxSemigroupGwrRn}]If $G$ is a finite group and $S=\Gwr R_n$, the wreath product of $R_n$ with $G$, then the Fourier transform of an arbitrary $\C$-valued function on $S$ may be computed in $O(|S|\log^4|S|)$ operations.
\end{thma}

We proceed as follows. In Section \ref{SecInvSemigroups} we review basic concepts related to semigroup theory and we list a few properties of the most important inverse semigroup, the rook monoid. 
In Section 
\ref{SecFourierTransforms}  
we review some representation theory related to inverse semigroups, we define the notion of the Fourier transform on an inverse semigroup, and we make precise the notion of a fast Fourier transform. 
In Section \ref{secGeneralApproach} we create a general framework for constructing FFTs on finite inverse semigroups. In particular, we explain how the problem of computing the Fourier transform on an inverse semigroup may be reduced to the problems of computing Fourier transforms on its maximal subgroups and a zeta transform on its poset structure. 
In Section \ref{SecRookFFT} we proceed according to our general framework to construct an FFT for the rook monoid. In Section \ref{SecWreath} we generalize this FFT to an FFT for wreath products of the rook monoid with arbitrary finite groups. 
Section \ref{SecFutureDirections} contains thoughts on future directions for this line of research.

\section{Inverse Semigroups}
\label{SecInvSemigroups}

A {\em semigroup} $S$ is a nonempty set with an associative, binary operation, which we will write multiplicatively. If there is an identity element for the multiplication in $S$, then $S$ is said to be a {\em monoid}. A group is a monoid where every element has a (unique) multiplicative inverse.
In this paper we are concerned with the class of semigroups known as {\em inverse semigroups}.

\begin{defn}An {\em inverse semigroup} is a semigroup $S$ such that, for each $x \in S$, there is a {\em unique} $y \in S$ such that
\[xyx = x \textup{ and } yxy = y. \] In this case, we write $y=x^{-1}$. 
\end{defn}

We remark that the condition that $y$ be unique is necessary in this definition. An element $x \in S$ is said to be {\em regular} or {\em Von-Neumann regular} if there is at least one $y \in S$ satisfying $xyx = x$ and $yxy = y$, and $S$ is said to be {\em regular} if every element of $S$ is regular. Consider the full transformation semigroup $T_n$ on the set $\{1, 2, \ldots, n\}$, that is, all maps from $\{1, 2, \ldots, n\}$ to itself under composition. It is easy to see that $T_n$ is regular, and that (for $n \geq 2$) there exist elements $x \in T_n$ for which there are {\em multiple} elements $y \in X$ satisfying $xyx = x$ and $yxy = y$. $T_n$ is therefore not inverse.  An equivalent characterization of inverse semigroups (see, e.g., \cite{Lawson}) is as follows.

\begin{thm}An inverse semigroup is a semigroup $S$ which is regular and in which all idempotents of $S$ commute.\end{thm}

The most important finite inverse semigroup is the {\em rook monoid} (also called the {\em symmetric inverse semigroup}) on $n$ elements, which we denote by $R_n$. It is the semigroup of all injective partial functions from $\{1,\ldots,n\}$ to itself under the usual operation of partial function composition. In this paper, we adopt the convention that maps act on the left of sets, and so, for $g,f\in R_n$, $g\circ f$ is defined for precisely the elements $x$ for which $x \in \dom (f)$ and $f(x) \in \dom (g)$. $R_n$ is called the rook monoid because it is isomorphic to the semigroup of all $n \times n$ matrices with the property that at most one entry in each row is 1 and at most one entry in each column is 1 (the rest being 0) under matrix multiplication, and such matrices (called {\em rook matrices}) correspond to the set of all possible placements of non-attacking rooks on an $n \times n$ chessboard. For example, consider the element $\sigma \in R_4$ defined by
\[\sigma(2)=1, \quad \sigma(4)=4.             \]
Then, viewed as a partial permutation, $\sigma$ is
\begin{displaymath}
\sigma = \left( \begin{array}{cccc}
1&2&3&4\\
-&1&-&4 \end{array} \right)
\end{displaymath}
where the dash indicates that the above entry is not mapped to anything. As a rook matrix, we have
\begin{displaymath}
\sigma = \left[
\begin {array}{cccc}
0&1&0&0\\
0&0&0&0\\
0&0&0&0\\
0&0&0&1
\end {array}
\right]. \end{displaymath} 

It is easy to see that $R_n$ is indeed an inverse semigroup, as the unique semigroup-inverse of a rook matrix is its transpose.


The rook monoid is a generalization of the symmetric group (as the symmetric group $S_n$ is contained in $R_n$ as the set of elements with nonzero determinant), and it plays the same role for finite inverse semigroups as the symmetric group does for finite groups in the following variation of Cayley's theorem \cite[p. 36--37]{Lawson}.
\begin{thm}Let $S$ be a finite inverse semigroup. Then there exists $n\in \Z$ for which $S$ is isomorphic to a subsemigroup of $R_n$. 
\end{thm}

It is often useful to consider subsets of $R_n$ whose elements have domains of a certain size.
\begin{defn}Given an element $\sigma \in R_n$, the {\em rank}\index{rank} of $\sigma$, denoted \textup{rk}$(\sigma)$, is defined to be \textup{rk}$(\sigma) = |\dom(\sigma)| = |\ran(\sigma)|$. It is clear that the rank of $\sigma$ is the same as the rank of the associated rook matrix.
\end{defn}

We have the following theorem on the size of $R_n$.
\begin{thm}
\label{ThmSizeRn}
\[|R_n| = \sum_{k=0}^{n} {\binom{n}{k}}^2 k!\]
\end{thm}
\begin{proof}
For any particular rank $k$, there are $\binom{n}{k}$ choices for the domain and $\binom{n}{k}$ choices for the range of an element of $R_n$, and for any particular choice of domain and range, there are $k!$ ways of mapping the domain to the range.\end{proof}

We also have the recursive formula, which we will prove in Section \ref{SecRookFFT}.
\begin{thm}
\index{rook monoid!cardinality}
\label{SizeRnRecursiveThm}
For $n \geq 3$, 
\[
|R_n| = 2n|R_{n-1}| - (n-1)^2|R_{n-2}|.   
\]
\end{thm}

\section{Fourier transforms for Inverse Semigroups}
\label{SecFourierTransforms}

For the rest of this paper, $S$ will denote a finite inverse semigroup. 

\begin{defn}The {\em semigroup algebra} of $S$ over $\C$, denoted $\C S$, is the formal $\C$-span of the symbols $\{s\}_{s\in S}$. Multiplication in $\C S$, denoted by $\ast$, is given by convolution (i.e., the linear extension of the semigroup operation via the distributive law): Suppose $f, g \in \C S$, with
\[
f = \sum_{r\in S}f(r) r, \quad g=\sum_{t\in S}g(t) t.
\]
Then
\[
f\ast g = \sum_{r\in S}f(r) r \sum_{t\in S}g(t) t = \sum_{s \in S} \sum_{r,t \in S: rt=s} f(r) g(t) s.
\]
\end{defn}
\noindent {\bf Remark:} If $S$ is a group, then convolution may be written in the familiar way:
\[f\ast g = \sum_{s \in S}\sum_{r \in S} f(r)g(r^{-1}s)s.\]

Let $f:S\rightarrow \C$ be any complex-valued function on $S$. We view $f$ as an element of the semigroup algebra $\C S$ by associating the natural basis of $\C S$ with the characteristic functions of the elements $s\in S$. So
\[
f = \sum_{s\in S}f(s)\delta_s
\] 
corresponds to
\[
\sum_{s\in S}f(s)s \in \C S.
\]
Thus, $\C S$ is the algebra of all $\C$-valued functions on $S$. 
The natural basis $\{s\}_{s\in S}$ of $\C S$ is also called the {\em semigroup basis}.


\begin{defn}A {\em representation} $\rho$ (of dimension $d_\rho \in \N$) of the semigroup algebra $\C S$ is an algebra homomorphism
\[
\rho: \C S\rightarrow M_{d_\rho}(C).
\]
\end{defn}
Equivalently, a representation $\rho$ of $\C S$ is a $d_\rho$-dimensional $\C$-vector space which is also a left $\C S$-module.

\begin{defn}A representation $\rho$ of $\C S$ is said to be {\em null} if $\rho(a)$ is the zero matrix for all $a\in \C S$.
\end{defn}

\begin{defn}A representation $\rho$ of $\C S$ is said to be {\em irreducible} if it is non-null and simple as a left $\C S$-module. That is, $\rho$ is irreducible if $\rho \neq 0$ and there is no $4$-tuple $(X,\rho_1,\rho_2,g)$, where $X$ is an invertible matrix, $\rho_1$ and $\rho_2$ are representations of $\C S$, and $g$ is a matrix-valued function on $\C S$, for which
\[
X \rho(a) X^{-1}= \left(\begin{array}{cc}\rho_1(a) & 0 \\ g(a) & \rho_2(a) \\ \end{array} \right)
\]
for all $a\in \C S$. 
\end{defn}

\begin{defn}
Two representations $\rho_1$ and $\rho_2$ of $\C S$ are {\em equivalent} if there is an invertible matrix $X$ for which
\[
X\rho_1(a)X^{-1} = \rho_2(a)
\]
for all $a\in \C S$. That is, two representations are equivalent if they are isomorphic as left $\C S$-modules.
\end{defn}

Let $S$ be a finite inverse semigroup. In \cite[Theorem 4.4]{Munn2}, Munn proves that the semigroup algebra $\C S$ is semisimple, and we therefore have: 
\begin{thm}Any representation of $\C S$ is equivalent to a direct sum of irreducible and null representations of $\C S$.
\end{thm}

Furthermore, Wedderburn's theorem applies to $\C S$.
\begin{thm}[Wedderburn's theorem] Let $S$ be a finite inverse semigroup. Let $\Y$ be a complete set of inequivalent, irreducible representations of $\C S$. Then $\Y$ is finite, and the map
\begin{equation}
\bigoplus_{\rho \in \Y}\rho:\C S\rightarrow \bigoplus_{\rho \in \Y} M_{d_\rho}(\C)
\label{WedderSemi}
\end{equation}
is an isomorphism of algebras. Explicitly, let $f\in \C S$ where
$
f=\sum_{s\in S}f(s)s.
$
Then
\[
f \mapsto \bigoplus_{\rho \in \Y}\sum_{s\in S}f(s)\rho(s) 
\]
in this isomorphism.
\end{thm}

We also have the following formula for the sum of the squares of the dimensions of the irreducible representations of $S$.
\begin{cor}
\label{CorSumDim}
Let $S$ be a finite inverse semigroup, and let $\Y$ be a complete set of inequivalent, irreducible representations of $\C S$. Then
\begin{equation}
|S| = \sum_{\rho\in\Y} d_\rho^2.  
\label{Sum of Squares of Dimensions}
\end{equation}
\end{cor}
\begin{proof}
The formula (\ref{Sum of Squares of Dimensions}) is just the $\C$-dimensionality of the algebras appearing in (\ref{WedderSemi}).
\end{proof}


Let $S$ be a finite inverse semigroup and let $f\in \C S$ with
\[
f=\sum_{s\in S}f(s)s. 
\]

\begin{defn}If $\rho$ is a representation of $\C S$, then we define the {\em Fourier transform of $f$ at $\rho$}, denoted by $\hat f(\rho)$, by
\[
\hat f(\rho) = \rho(f) = \sum_{s\in S}f(s)\rho(s).
\]
\end{defn}

Let $\Y$ be a complete set of inequivalent, irreducible representations of $\C S$. The map in Wedderburn's theorem obtained by ``gluing" together the elements of $\Y$ is called a {\em Fourier transform on} (or {\em for}) $S$. Specifically, we have
\begin{defn}The element of $\bigoplus_{\rho\in\Y}M_{d_\rho}(\C)$ given by
\[
\bigoplus_{\rho \in \Y}\sum_{s\in S}f(s)\rho(s)
\]
is called the {\em Fourier transform of $f$ with respect to (or relative to) $\Y$}.
\end{defn}


Consider the inverse image of the natural basis of the algebra on the right in (\ref{WedderSemi}), that is, the inverse image of the set of matrices in the algebra on the right having exactly one entry equal to $1$ (the rest being $0$). This is the target basis for the Fourier transform on $\C S$. Each block of the algebra on the right is a minimal two-sided ideal. The inverse image of the basis for a single block is a therefore a basis for a minimal two-sided ideal of $\C S$. Since the map in (\ref{WedderSemi}) is an isomorphism, the inverse images of distinct columns have intersection $\{0\}$, and this basis for $\C S$ therefore realizes the decomposition of $\C S$ into the direct sum of its minimal ideals. 

\begin{defn}Let $S$ be a finite inverse semigroup and let $\Y$ be a complete set of inequivalent, irreducible representations of $\C S$. The inverse image of the natural basis of the algebra on the right in (\ref{WedderSemi}) is called a {\em Fourier basis} for $\C S$. More specifically, it is the {\em Fourier basis for $\C S$ according to $\Y$}. It is also known as the {\em dual matrix coefficient basis for $\C S$ relative to $\Y$} \cite{Maslen}.
\end{defn}

When we refer to a Fourier basis for $\C S$, we mean any basis of $\C S$ that can arise in this manner by choosing an appropriate set of representations $\Y$. To see why we consider this is a ``normalization" condition, consider $S=\ZnZ$. Every irreducible representation of $\C \ZnZ$ has dimension 1, so the notion of equivalence of representations for $\C \ZnZ$ reduces to the notion of equality.
Specifically, the irreducible representations of $\C \ZnZ$ are the $\chi_k$ ($k=0,1,\ldots, n-1$), given on the natural basis $\ZnZ$ by
\[
\chi_k({j}) = e^{\frac{-2\pi i j k }{n}}.
\]
In this case, the isomorphism (\ref{WedderSemi}) is the usual discrete Fourier transform:
\[
\hat{f}(\chi_k) = \sum_{j=0}^{n-1} f(j)e^{\frac{-2\pi i j k }{n}},
\]
and the associated Fourier basis of $\C \ZnZ$ is the usual basis of exponential functions $\{b_k\}_{k=0}^{n-1}$, normalized so that $\hat{b_k}(\chi_k)=1$:
\[
b_k=\frac{1}{n}\sum_{j=0}^{n-1}e^{\frac{2\pi i j k}{n}} j.
\]




Here is the general convolution theorem.

\begin{thm}The Fourier transform on $S$ turns convolution into multiplication of block-diagonal matrices. The Fourier transform turns convolution into pointwise multiplication if and only if every irreducible representation of $\C S$ has dimension one.
\end{thm}

\begin{proof}Since the map given in Wedderburn's theorem is a homomorphism, it turns multiplication in $\C S$ (that is, convolution) into multiplication in $\bigoplus_{\rho\in\Y}M_{d_\rho}(\C)$.
\end{proof}

We begin our study of the computational complexity of the Fourier transform on $S$ by introducing some notation.

\begin{defn}Let $\Y$ be a complete set of inequivalent, irreducible representations of $\C S$. Suppose that all representations in $\Y$ are precomputed (that is, evaluated at every element of some basis of $\C S$---usually the standard $\{s\}_{s\in S}$ basis) and stored in memory. Then the maximum number of operations (where an operation is defined to be a complex multiplication together with a complex addition) needed to compute the Fourier transform of an arbitrary function $f=\sum_{s\in S}f(s)s \in \C S$ is denoted by
\[
\T_{\Y}(S).
\]
Now, let $\Y$ vary over all sets of inequivalent, irreducible representations of $\C S$. We define
\[
\Cplx(S)=\min_{\Y}\{\T_\Y(S)\}.
\]
\end{defn}

Given a particular inverse semigroup $S$, the goal is to bound $\Cplx(S)$. This is often accomplished by constructing a computationally advantageous set of representations $\Y$ of $\C S$ and proving a bound on $\T_\Y(S)$ that compares favorably to the number of operations needed to compute the Fourier transform by naive methods, as given in the following theorem.

\begin{thm}For any inverse semigroup $S$, $\Cplx(S)\leq |S|^2.$\end{thm}
\begin{proof}Let $\Y$ be any complete set of inequivalent, irreducible representations of $\C S$. If $f=\sum_{s\in S}f(s)s\in \C S$, then a naive computation of ($\ref{WedderSemi}$) requires at most
\[
\sum_{\rho\in\Y}|S|d_\rho^2 = |S|\sum_{\rho\in\Y}d_\rho^2 = |S|^2
\]
operations, the last equality arising from corollary \ref{CorSumDim}.
\end{proof}


A {\em fast Fourier transform} (FFT) on (or for) an inverse semigroup $S$ is then an algorithm for computing the Fourier transform of an arbitrary $\C$-valued function on $S$ whose computational complexity compares favorably to that of the naive algorithm.

\section{FFTs for Inverse Semigroups---A General Approach}
\label{secGeneralApproach}
In this section, we explain a theorem of B. Steinberg \cite[Theorem 4.6]{Steinberg2} and we use it to reduce the problem of creating FFTs on finite inverse semigroups to the problems of creating FFTs on their maximal subgroups and fast zeta transforms on their poset structures. 

\subsection{The Groupoid Basis}

Let $S$ be a finite inverse semigroup. We begin by recalling the {\em natural poset structure} of 
 $S$ \cite{CliffandPres,Lawson,Steinberg2}.

\begin{defn}Let $S$ be a finite inverse semigroup. For $s, t \in S$, define
\begin{align*}s \leq t &\iff s = et \textup{ for some idempotent } e\in S \cr 
& \iff s = tf \textup{ for some idempotent } f\in S. 
\end{align*}
\end{defn}
For $R_n$, the idempotents are the restrictions of the identity map, and this ordering is the same as the ordering
\[
s\leq t \iff \up{$t$ extends $s$ as a partial function}.
\]
\noindent {\bf Remark:} If $S$ is a group, then its poset structure is trivial in the sense that $s\leq t \iff s=t$. 

If $P$ is a finite poset, then the {\em zeta function}\index{zeta function}  $\zeta$ of $P$ is given by:
\[
\zeta:P\times P \rightarrow \{0,1\}
\]
\[ \zeta(x,y) =
\begin{cases}
1 & \textup{ if }x \leq y, \\
0 & \textup{ otherwise.}\end{cases}
\]

The zeta function is invertible over any field $F$ (in fact, over any ring with unity), and its inverse is called the {\em M\"obius function}\index{M\"obius function}  $\mu$. The M\"obius function for $R_n$ over $\C$ is well known \cite{Stanley,Steinberg2}. It is, for $x\leq y$, 
\[\mu(x,y) = (-1)^{\rk(y)-\rk(x)}.\]
We have already seen one natural basis for the semigroup algebra $\C S$, the semigroup basis $\{ s \}_{s\in S}$. Multiplication in $\C S$ with respect to this basis is just the linear extension of the multiplication in $S$. In \cite{Steinberg2}, B. Steinberg defines another ``natural'' basis of $\C S$. To motivate this new basis, recall that every finite inverse semigroup is isomorphic to a subsemigroup of a rook monoid and can therefore be viewed as a collection of partial functions. There is another model for composing partial functions---only allow the composition if the range of the first function ``lines up'' with the domain of the second. For example, if
\[
\sigma=\left(
\begin{array}{cccc}
1&2&3&4\\
2&-&1&-\\
\end{array}
\right),
\qquad
\pi=\left(
\begin{array}{cccc}
1&2&3&4\\
4&3&-&-\\
\end{array}
\right),
\]
then the idea is that the composition $\pi \circ \sigma$ is
\[
\pi \circ \sigma =
\left(
\begin{array}{cccc}
1&2&3&4\\
3&-&4&-\\
\end{array}
\right),
\]
and the composition $\sigma \circ \pi$ is disallowed. The {\em groupoid basis} \index{groupoid basis} of $\C S$ encodes this.

\begin{defn}Define, for each $s \in S$, the element $\ld s \rd \in \C S$ by
\[\ld s \rd = \sum_{t\in S: t \leq s} \mu(t,s) t.\]
The collection $\{\s\}_{s\in S}$ is called the {\em groupoid basis} of $\C S$.
\end{defn}

Viewing $S$ as a subsemigroup of $R_n$, then, we have

\begin{thm} The groupoid basis is a basis for $\C S$. Multiplication in $\C S$ relative to this basis is given by the linear extension of
\[
\ld s \rd \ld t \rd = 
\begin{cases}
\ld st \rd & \textup{ if }\dom (s) = \ran (t), \\
0 & \textup { otherwise.}
\end{cases}
\]
Furthermore, the change of basis to the $\{s\}_{s\in S}$ basis of $\C S$ is given by M\"obius inversion:
\begin{equation}
\label{sissumofbrackett}
s = \sum_{t \in S: t \leq s} \ld t \rd.
\end{equation}
\end{thm}

\begin{proof}
This is \cite[Lemma 4.1 and Theorem 4.2]{Steinberg2}, using our convention that maps act on the left of sets.
\end{proof}

The notions of $\dom(s)$ and $\ran(s)$ may also be defined intrinsically in terms of $S$, i.e., without reference to an embedding of $S$ into $R_n$. Specifically, for any element $s\in S$, $ss^{-1}$ and $s^{-1}s$ are idempotent, and one may define \index{domain}\index{range}
\begin{align*}
&\dom(s)=s^{-1}s \\
&\ran(s)=ss^{-1}.
\end{align*}
If we use this definition for $s\in R_n$, we see that $\dom(s)$ is actually {\em not} the domain of $s$, but is rather the map which is the identity on the domain of $s$ and undefined elsewhere, and likewise for $\ran(s)$. This means that we are abusing the distinction between the domain and range of a map and the corresponding partial identities. Under this definition, we have that the groupoid basis of $\C S$ multiplies as follows:
\begin{equation}
\label{GroupoidMult}
\ld s \rd \ld t \rd = 
\begin{cases}
\ld st \rd & \textup{ if }s^{-1}s = tt^{-1}, \\
0 & \textup { otherwise.}
\end{cases}
\end{equation}

The viewpoint, then, is that we have two ``natural bases'' for $\C S$, the semigroup basis $\{s\}_{s \in S}$ and the groupoid basis $\{\ld s\rd \}_{s \in S}$. (However, note that if $S$ is a group, then $s = \s \in \C S$ for all $s\in S$, so group algebras only have one natural basis).

Our goal is to construct an efficient change of basis from the semigroup basis of $\C S$ to a Fourier basis. We will do so by constructing an efficient change of basis from the semigroup basis of $\C S$ to the groupoid basis of $\C S$ (that is, a {\em fast zeta transform} on $S$) and an efficient change of basis from the groupoid basis of $\C S$ to a Fourier basis, so that the composition of these two changes of basis will be an FFT for $S$. We will focus on the second of these changes of basis first. In order to do so, we must first understand how the groupoid basis realizes the decomposition of $\C S$ into a direct sum of matrix algebras over group algebras.

\subsection{Matrix Algebras Over Group Algebras}

Given an element $s\in S$, it is natural to think of $s$ as an ``isomorphism'' from $\dom(s)$ to $\ran(s)$. As in \cite{Steinberg2}, we use this to define the notion of {\em isomorphic idempotents}. 
\begin{defn}If $a,b\in S$ are idempotent, then $a$ and $b$ are said to be {\em isomorphic} idempotents if there is an ``isomorphism'' from $a$ to $b$, that is, if there is an element $s\in S$ such that $a=s^{-1}s$ and $b=ss^{-1}$. 
\end{defn}
Let us now define two idempotents in $S$ to be {\em $\D$-related} if they are isomorphic. For the rook monoid $R_n$, the idempotents are the restrictions of the identity map, and two idempotents are isomorphic if and only if they have the same rank. We extend $\D$ to an equivalence relation on $S$ by defining $s\D t$ if $s^{-1}s$ is isomorphic to $t^{-1}t$ (or, equivalently, if $ss^{-1}$ is isomorphic to $tt^{-1}$). This is Green's famous $\D$-relation \cite{CliffandPres,Green}, and the equivalence classes of $S$ with respect to this relation are called the $\D$-classes of $S$. We mention that an equivalent characterization of $\D$ is that $s\D t$ if and only if $s$ and $t$ generate the same two-sided ideal in $S$. 
For $R_n$, there are $n+1$ $\D$-classes. They are $D_0,D_1,\ldots,D_n$, where $D_k$ is the set of elements of $R_n$ of rank $k$.

Let $e\in S$ be idempotent. Let $G_e$ be the {\em maximal subgroup at $e$}\index{maximal subgroup}, that is, the largest subset of $S$ which contains $e$ and which is also a group. It is easy to see that
\[
G_e = \{s\in S: s^{-1}s=ss^{-1}=e\},
\]
and that $e$ is the identity of $G_e$. If $a,b$ are isomorphic idempotents, it is straightforward to show that $G_a\cong G_b$. For $R_n$, the maximal subgroup at an idempotent $e$ of rank $k$ is isomorphic to the permutation group $S_k$.

Now, let us describe the decomposition of the semigroup algebra $\C S$ into a direct sum of matrix algebras over group algebras. This is B. Steinberg's result \cite[Theorem 4.5]{Steinberg2}, and we include the proof because the construction of the isomorphism involved is important in the construction of the FFTs to come. Let $D_0,\ldots,D_n$ be the $\D$-classes of $S$. Let $\C D_k$ be the $\C$\up{-}\textup{span} of $\{\s:s\in D_k\}$. It is immediate from (\ref{GroupoidMult}) that $\C S=\bigoplus_{k=0}^n\C D_k$.
\begin{thm}[B. Steinberg]
\label{BigThm}
Let $r_k$ indicate the number of idempotents in $D_k$, and let $e_k$ be any idempotent in $D_k$. Denote the maximal subgroup of $S$ at $e_k$ by $G_k$. Then, as algebras,
$\C D_k\cong M_{r_k}(\C G_k).$
\end{thm}
\begin{proof}
We already know that $G_a$ and $G_b$ are isomorphic for any idempotents \mbox{$a,b\in D_k$}. Now, fix an idempotent $e_k\in D_k,$ and for every idempotent $a\in D_k$, fix an element $p_a\in S$ such that ${p_a}^{-1}p_a=e_k$ and $p_a{p_a}^{-1}=a$ (that is, $p_a$ is an isomorphism from $e_k$ to $a$). 
Let us take $p_{e_k}=e_k$. It is easy to show that, in fact, $p_a\in D_k$. We view our $r_k\times r_k$ matrices as being indexed by pairs of idempotents in $D_k$. We now define our isomorphism by defining it on the basis $\{\s:s\in D_k\}$ of $\C D_k$ and extending linearly. So, for an element $\s\in \C D_k$ with $s^{-1}s=a$ and $ss^{-1}=b$, define
\[
\phi(\s) = {p_b}^{-1}sp_{a} E_{b,a},
\]
where $E_{b,a}$ is the standard $r_k\times r_k$ matrix with a 1 in the $b,a$ position and 0 elsewhere.

A quick calculation shows that ${p_b}^{-1}sp_a \in G_k$ by construction. It is straightforward to show that $\phi$ is an isomorphism, with the inverse induced by, for $s\in G_k$,
\[
sE_{b,a} \mapsto \ld p_b s {p_a}^{-1}  \rd.
\]
\end{proof}

The corollary is:
\begin{cor}$\C S\cong\bigoplus_{k=0}^n M_{r_k}(\C G_k)$.
\end{cor}
A dimensionality count thus establishes
\begin{equation}
\label{SizeOfSMaxSubgps}
|S| = \sum_{k=0}^n r_k^2 |G_k|.
\end{equation}

Since we will construct an FFT for the rook monoid $R_n$ in Section \ref{SecRookFFT}, for clarity's sake we now explain what the isomorphism constructed in the proof of Theorem \ref{BigThm} translates into when $S=R_n$. Fix a $\D$-class $D_k$ (that is, the subset of elements of $R_n$ of rank $k$), and let us take $e_k \in D_k$ to be the partial identity on $\{1,\ldots,k\}$, that is
\[
e_k=\left(\begin{array}{ccccccc}
1&2&\cdots&k&k{+}1&\cdots&n \\
1&2&\cdots&k&-&\cdots&-
\end{array}\right).
\] 
We then have
\[
G_k=\{s\in R_n:\dom(s)=\ran(s)=\onetok\}.
\]
Let us identify $G_k$ with the permutation group $S_k$ in the obvious manner.

For an idempotent $a \in D_k$ (that is, a rank-$k$ restriction of the identity map), let us take $p_a$ to be the unique order-preserving bijection from $\onetok$ to \mbox{$\dom(a)=\ran(a)$}. For an element $s\in R_n$ of rank $k$, let us define the {\em permutation type}\index{permutation type} of $s$, $\perm(s)$, to be, informally, the ``arrows'' from $\dom(s)$ to $\ran(s)$, expressed as a permutation in $G_k=S_k$. For example, if 
\[
s=\left(\begin{array}{cccc}
1&2&3&4\\
4&-&1&2
\end{array}\right),
\textup{ then }
\perm(s) = \left(\begin{array}{ccc}
1&2&3\\
3&1&2
\end{array}\right)
\]
because $s$ sends the first element of its domain to the third element of its range, the second element of its domain to the first element of its range, and the third element of its domain to the second element of its range.

Formally, we define
\[
\perm(s)= {p_{\ran(s)}}^{-1} s p_{\dom(s)},
\]
where $\dom(s)$ and $\ran(s)$ are once again understood to be the corresponding partial identities in $R_n$, so that $p_{\dom(s)}$ is the unique order preserving bijection from $\onetok$ to $\dom(s)$ and ${p_{\ran(s)}}^{-1}$ is the unique order preserving bijection from $\ran(s)$ to $\onetok$.

The isomorphism $\phi$ defined in the proof of Theorem \ref{BigThm} now works as follows. We have $\binom{n}{k}\times \binom{n}{k}$ matrices, so let us index their rows and columns by the $k$-subsets of $\oneton$. We have
\begin{align*}
\C D_k &\cong M_{\binom{n}{k}}(\C S_k)\\
\textup{by } \phi(\s) &= \perm(s)E_{\ran(s),\dom(s)}.
\end{align*}

Therefore, we have:
\begin{cor}
\label{CorRnIsomSks}
$\C R_n\cong \bigoplus_{k=0}^n M_{\binom{n}{k}}(\C S_k).$
\end{cor}

\noindent{\bf Remark:} This result was implicit in the work of Munn \cite{Munn3} and was first written down explicitly by Solomon \cite{Solomon}. Solomon's isomorphism is essentially the same as the one given here.

\subsection{FFTs for Inverse Semigroups}

We can now give a bound on the number of operations needed to change from the groupoid basis of $\C S$ to a Fourier basis of $\C S$. 

\begin{thm}
\label{ThmGroupoidToFourier}
Let $S$ be a finite inverse semigroup with $\D$-classes $D_0,\ldots,D_n$. For each $\D$-class $D_k$, choose an idempotent $e_k$, and let $G_k$ be the maximal subgroup of $S$ at $e_k$. Let $r_k$ denote the number of idempotents in $D_k$. Let $v\in \C S$ be given with respect to the groupoid basis, that is
\[
v=\sum_{s\in S}v(s)\s.
\]
Then the number of operations needed to compute the Fourier transform of $v$ is no more than
\[
\sum_{k=0}^n r_k^2 \Cplx(G_k).
\]
\end{thm}

\begin{proof}
For each idempotent $a\in D_k$, fix an element $p_a\in S$ such that $p_a^{-1}p_a = e_k$ and $p_ap_a^{-1}=a$ (and take $p_{e_k}=e_k$). By the proof of Theorem \ref{BigThm}, this defines the isomorphism
\[
\C S\cong \bigoplus_{k=0}^n M_{r_k}(\C G_k).
\]
Let $\IRR(G_k)$ be any set of inequivalent, irreducible matrix representations of $\C G_k$. It is clear from this isomorphism that the irreducible representations of $\C S$ are in one-to-one correspondence with $\biguplus_{k=0}^n \IRR(G_k)$. Specifically, given a representation $\rho$ of $\C G_k$, we tensor it up to a representation of $M_{r_k}(\C G_k)$ and extend to $\C S$ by letting it be zero on the other summands and following the isomorphism. The resulting representation $\bar\rho$ of $\C S$ is thus given by the linear extension of
\[
\bar\rho(\s)=\begin{cases}
\bar\rho({p_{ss^{-1}}}^{-1}sp_{s^{-1}s}E_{{ss^{-1}},{s^{-1}s}}) = E_{{ss^{-1}},{s^{-1}s}}\otimes\rho({p_{ss^{-1}}}^{-1}sp_{s^{-1}s}) & \textup{if }s\in D_k, \\
0 & \textup{otherwise.}
\end{cases}
\]
Furthermore, the collection $\Y=\{\bar\rho:\rho\in\biguplus_{k=0}^n \IRR(G_k)\}$ forms a complete set of inequivalent, irreducible matrix representations of $\C S$. We will use this set of representations to obtain our bound. Let $\bar\rho \in \Y$, and suppose $\bar\rho$ was obtained by tensoring up a representation $\rho$ in $\IRR(G_k)$. Then
\[
\hat v(\bar\rho) = \bar\rho(v) = \sum_{s\in S}v(s)\bar\rho(\s) = \sum_{s\in D_k}v(s)\bar\rho(\s),
\]
the last equality arising from the fact that $\bar\rho$ is identically zero off of $\C D_k$. Now, let us view $\bar\rho(v)$ as an $r_k \times r_k$ matrix with entries in $d_\rho \times d_\rho$ matrices (so we are viewing the rows and columns of $\bar\rho(v)$ as indexed by the idempotents in $D_k$). For idempotents $a,b \in D_k$, the $b,a$ entry of $\bar\rho(v)$ is then the $d_\rho \times d_\rho$ matrix
\[
\bar\rho(v)_{b,a} = \sum_{\substack{s\in D_k:\\ss^{-1}=b\\s^{-1}s=a}}v(s)\rho({p_b}^{-1}sp_a).
\]
By the proof of Theorem \ref{BigThm}, this is the same as
\begin{equation}
\sum_{s\in G_k}v(p_b s {p_a}^{-1})\rho(s).
\label{GroupoidFFT}
\end{equation}
If we define a function $h_{b,a}:G_k\rightarrow \C$ by 
\[
h_{b,a}(s)=v(p_b s {p_a}^{-1}), 
\]
we see that (\ref{GroupoidFFT}) is just the Fourier transform of the function $h_{b,a}$ on the group $G_k$ at $\rho$. Notice that this holds regardless of the choice of $\IRR(G_k)$. Furthermore, once $\IRR(G_k)$ is chosen, it is clear from the above argument that the collection \mbox{$\{ \widehat{h_{b,a}}(\rho):\rho\in \IRR(G_k)\}$} 
consists exactly of the blocks that compose $\{\hat v(\bar\rho): \rho\in \IRR(G_k)\}$. An algorithm for computing the Fourier transform of $v$ thus presents itself---for each $\D$-class $D_k$, run $r_k^2$ Fourier transforms on $G_k$, and then arrange the results into block form to construct the $\hat v(\bar\rho)$. The latter step can be done for free in our computational model because it requires no operations. Thus, the number of operations required to compute the Fourier transform of $v$ with respect to $\Y$ is no more than
\[
\sum_{k=0}^n r_k^2 \T_{\IRR(G_k)}(G_k).
\]
Since we can choose the $\IRR(G_k)$ at will, choosing them to be in their most computationally advantageous forms for computing Fourier transforms on the $G_k$ reduces our bound on the number of operations needed to compute the Fourier transform of $v$ by this approach to
\[
\sum_{k=0}^n r_k^2 \Cplx(G_k).
\]
\end{proof}

Let us denote by $\Cplx(\zeta_S)$ the maximum number of operations to perform a zeta transform on $\C S$ (that is, the maximum number of operations needed to re-express an arbitrary element of $\C S$ with respect to the groupoid basis of $\C S$, given its expression in terms of the semigroup basis). Our main result follows immediately.

\begin{thm}\label{MyBigThm}
Let $S$ be a finite inverse semigroup with $\D$-classes $D_0,\ldots,D_n$. For each $\D$-class $D_k$, choose an idempotent $e_k$, and let $G_k$ be the maximal subgroup of $S$ at $e_k$. Let $r_k$ denote the number of idempotents in $D_k$. Then
\[
\Cplx(S) \leq \Cplx(\zeta_S) + \sum_{k=0}^n r_k^2 \Cplx(G_k).
\]
\end{thm}

The following corollary concerns the case when ``good" FFTs (that is, $c|G|\log^d|G|$-complexity FFTs) are known for every maximal subgroup $G$ of $S$.

\begin{cor}Suppose $S$ is a finite inverse semigroup with $\D$-classes $D_0,\ldots,D_n$. Let $r_k$ be the number of idempotents in $D_k$. Choose an idempotent $e_k$ from each $\D$-class $D_k$, and let $G_k$ be the maximal subgroup at $e_k$. If $\Cplx(G_k) \leq c_k|G_k| \log^{d_k} |G_k|$ for all $k$, then for some constants $c,d$ we have
\[
\Cplx(S) \leq \Cplx(\zeta_S) + c|S|\log^d|S|.
\]
\end{cor}
\begin{proof}
Let $c=\max_{k=0}^n c_k$, and let $d=\max_{k=0}^n d_k$. Then
\begin{align*}
\Cplx(S) & \leq \Cplx(\zeta_S) + \sum_{k=0}^n r_k^2 c |G_k|\log^d|G_k| \\
& \leq \Cplx(\zeta_S) + c\log^d|S| \sum_{k=0}^n r_k^2|G_k| \\
& = \Cplx(\zeta_S) + c|S|\log^d|S|,
\end{align*}
the last equality arising from (\ref{SizeOfSMaxSubgps}).
\end{proof}

Thus, the problem of constructing a fast Fourier transform for a finite inverse semigroup $S$ may be solved by constructing a fast zeta transform on the poset structure of $S$ and constructing fast Fourier transforms for all of the maximal subgroups of $S$---in particular, an $O(|S|\log^d|S|)$ zeta transform for $S$ together with $O(|G|\log^d|G|)$ FFTs for each of the maximal subgroups $G$ of $S$ combine to form an $O(|S|\log^d|S|)$ FFT for $S$.

\noindent {\bf Remark:} Note that the approach outlined in this section does not make use of the assumption that a complete set of representations of $\C S$ is precomputed and stored in memory, and in fact full precomputation is unnecessary. To use the approach presented here, we only need to precompute a complete set of representations for each of the maximal subgroups $G_k$ of $S$ (the cost of which is handled in the $\Cplx (G_k)$ terms in Theorem \ref{MyBigThm}). Equivalently, once we have chosen one idempotent $e_k$ from each $\D$-class $D_k$ of $S$, we only need to precompute the set of representations $\Y$ used in the proof of Theorem \ref{ThmGroupoidToFourier} on the following subset of the groupoid basis of $\C S$:
\[
\bigcup_{k=0}^n\{\s:s\in G_k\}.
\]
This amounts to precomputing $\Y$ on the full groupoid basis, as the structure of $\rho(\s)$ can then be inferred for any $\rho\in\Y$ and any other groupoid basis element $\s$ as a simple permutation of blocks of one of the matrices already computed.

\noindent {\bf Remark:} The problem of creating a fast zeta transform on the poset structure of a finite inverse semigroup appears difficult to tackle in generality because the poset structure can be about as bad as one wants---every finite meet semilattice is a (commutative and idempotent) finite inverse semigroup under the meet operation, so it is at least possible to encounter any finite meet semilattice as the poset structure of an inverse semigroup. It remains to be seen whether there are any general principles one might employ when creating fast zeta transforms. In the next two sections, we develop specific fast zeta transforms for the poset structures of the rook monoid and its wreath product by arbitrary finite groups and combine them with known results on group FFTs for the symmetric group and its wreath products to obtain $O(|S|\log^d|S|)$ FFTs for these inverse semigroups.


\section{An FFT for the Rook Monoid}
\label{SecRookFFT}

We now use the approach from Section \ref{secGeneralApproach} to create an $O(|R_n| \log^3|R_n|)$-complexity Fourier transform for the rook monoid $R_n$. 
We begin by handling the term in Theorem \ref{MyBigThm} concerning the change of basis from the groupoid basis of $\C R_n$ to a Fourier basis.

\begin{thm}$\Cplx(R_n)\leq \Cplx(\zeta_{R_n}) + \frac{3}{4} n(n-1)|R_n|$.
\label{ThmRookPart1}
\end{thm}
\begin{proof}Let $D_0,D_1,\ldots,D_n$ be the $\D$-classes of $R_n$. Recall that $D_k$ is the elements of $R_n$ of rank $k$. Let $e_k\in D_k$ be the partial identity on $\{1,2,\ldots,k\}$, so the maximal subgroup $G_k$ of $R_n$ at $e_k$ is isomorphic to $S_k$.

Corollary \ref{CorRnIsomSks} and Theorem \ref{MyBigThm} then imply that
\[
\Cplx(R_n)\leq \Cplx(\zeta_{R_n}) + \sum_{k=0}^n\binom{n}{k}^2\Cplx(S_k).
\]

If we let $\Y_k$ denote a complete set of irreducible representations for $S_k$ given in Young's seminormal or orthogonal form (descriptions of which may be found in \cite{Clausen} or Chapter 3 of \cite{JamesAndKerber}), then we may use Maslen's FFT for the symmetric group \cite{Maslen} to obtain $\Cplx(S_k)\leq T_{\Y_k}(S_k)\leq \frac{3}{4}k(k-1)|S_k|$.

From here, we have
\begin{align*}
\sum_{k=0}^n\binom{n}{k}^2\Cplx(S_k)\leq &\sum_{k=0}^n \binom{n}{k}^2\frac{3}{4}k(k-1)|S_k| \\
&\leq\frac{3}{4} n(n-1) \sum_{k=0}^n \binom{n}{k}^2|S_k|\\
&=\frac{3}{4} n(n-1) \sum_{k=0}^n \binom{n}{k}^2k!\\
&=\frac{3}{4}n(n-1)|R_n|,
\end{align*}
where the last equality follows from Theorem \ref{ThmSizeRn}. The theorem follows.

\end{proof}

We now turn to analyzing $\Cplx(\zeta_{R_n})$. Let $f\in \C R_n$ be an arbitrary element, expressed with respect to the semigroup basis, that is
\[
f=\sum_{s\in R_n}f(s)s.
\]
We would like to express $f$ with respect to the groupoid basis, that is
\[
f=\sum_{s\in R_n}g(s)\s,
\]
where, by (\ref{sissumofbrackett}), the coefficients $g(s)$ are given by
\[
g(s)=\sum_{\substack{t\in R_n: \\ t\geq s}}f(t).
\]
Our goal is to compute the coefficients $g(s)$ in an efficient manner, and we give an algorithm below for doing so. First, however, we present the proof of Theorem \ref{SizeRnRecursiveThm}, as the algorithm we give below is based (at least in part) on the ideas involved in the proof.
\begin{thm}[Theorem \ref{SizeRnRecursiveThm}]
\label{ThmSizeRnRecursiveTwo}
For $n \geq 3$, 
\[
|R_n| = 2n|R_{n-1}| - (n-1)^2|R_{n-2}|.
\]
\end{thm}
\begin{proof}
Viewing the elements of $R_n$ as rook matrices, $R_n$ consists of those elements having all 0's in column $1$ and row $1$ (of which there are $|R_{n-1}|$), together with, for each $\alpha \in \{1,\ldots, n\}$, those having a 1 in position $(\alpha, 1)$ (of which there are $n|R_{n-1}|$ total), together with, for each $\alpha \in \{2,\ldots, n\}$, those having a 1 in position $(1,\alpha)$ (of which there are $(n-1)|R_{n-1}|$ total). Counting the number of elements of $R_n$ in this way overcounts. For each pair $\alpha,\beta$ with $2 \leq \alpha , \beta \leq n$, every element with ones in positions $(\alpha,1)$ and $(1,\beta)$ (of which there are $(n-1)^2|R_{n-2}|$ total) gets counted twice.
\end{proof}

We now explain the fast zeta transform, noting that the savings in time afforded by this algorithm come at the expense of a modest additional storage requirement over the naive algorithm---the algorithm presented here requires the storage of up to $O(n^{1/4}|R_n|)$ complex numbers in memory during runtime (see Theorems \ref{ThmZetaMemoryUpper} and \ref{ThmZetaMemoryLower}), as opposed to the naive algorithm, which requires at most $2|R_n|$. 

Let us denote $g(s)=\sum_{t\geq s}f(t)$ by $\zeta_f(s)$. The basic idea is to ``work from the top down.'' Since we are trying to compute $\zeta_f(s)$ for all $s\in R_n$, it makes sense to begin with the elements of rank $n$. If $\rk(s)=n$, then there is no element $t$ such that $t>s$, so $\zeta_f(s)=f(s)$, and this requires no operations. Next, if $\rk(s)=n-1$, then there is only one element $t\in R_n$ such that $t>s$, so
\[
\zeta_f(s) = f(s) + f(t) = f(s) + \zeta_f(t).
\]
Next, if $\rk(s)=n-2$, consider the poset consisting of the elements $t\in R_n$ for which $t\geq s$. This poset is isomorphic to the poset for $R_2$, with $s$ in the place of the $0$ element. We proceed down in rank in this manner, and the aim of this fast zeta transform is, in general, to find a way to re-use the $\zeta_f(t)$ we have already computed in order to efficiently compute $\zeta_f(s)$. In fact, instead of computing just $\zeta_f(s)$ for the elements $s$ of rank $k$, we compute $\zeta_f(s)$ along with $n-k$ other numbers for each element $s$ of rank $k$. These other numbers are needed for the efficient computation of the zeta transform at elements of rank $k-1$ and $k-2$, and can be discarded when they are no longer needed. We introduce some notation.

Let $s\in R_n.$ Then $s$ is a partial permutation of $\{1,2,\ldots,n\}$.
\begin{itemize}
\item Let $d_i(s)$ be the $i^{th}$ element of $\{1,2,\ldots,n\}$ {\em not} in $\dom(s)$.

\item Let $r_i(s)$ be the $i^{th}$ element of $\{1,2,\ldots,n\}$ {\em not} in $\ran(s)$.
\end{itemize}
That is, $d_i(s)$ is simply the $i^{th}$ element of the complement of the domain of $s$ (taken in order), and similarly for $r_i(s)$. Define ``partial'' zeta transforms at $s$ as follows:
\[
\zeta_f(s,\{d_1(s),d_2(s),\ldots,d_{m}(s)\},\{r_1(s),r_2(s),\ldots,r_{m}(s)\}) = \sum_{\substack{t\geq s: \\ d_1(s),\ldots,d_m(s)\notin \dom(t) \\r_1(s),\ldots,r_m(s)\notin \ran(t)}}f(t)
\]
Our zeta transform proceeds as follows, with steps $0, 1, \ldots, n$:
\begin{itemize}
\item Step 0: For all $s\in R_n$ with $\rk(s)=n$, compute all $\zeta_f(s,\{\},\{\})=\zeta_f(s)$  (0 operations).
\item Step 1: For all $s\in R_n$ with $\rk(s)=n-1$, compute $\zeta_f(s,\{\},\{\})=\zeta_f(s)$ and $\zeta_f(s,\{d_1(s),r_1(s)\})$ ($1$ operation for each element $s$).
\[
\vdots
\]
\item Step $n-k$: For all $s\in R_n$ with $\rk(s)=k$, compute all
\begin{align*}
& \zeta_f(s,\{\},\{\}) = \zeta_f(s), \\
& \zeta_f(s,\{d_1(s)\},\{r_1(s)\}), \\
& \zeta_f(s,\{d_1(s),d_2(s)\},\{r_1(s),r_2(s)\}), \\
& \vdots \\
& \zeta_f(s,\{d_1(s),d_2(s),\ldots,d_{n-k}(s)\},\{r_1(s),r_2(s),\ldots,r_{n-k}(s)\}).
\end{align*}
\[
\vdots
\]
\end{itemize}

\begin{thm}Step $n-k$ requires at most
\[
\left((n-k)^2 + \frac{(n-k-1)(n-k)(2n-2k-1)}{6}\right)
\binom{n}{k}^2 k! 
\]
operations in total.
\end{thm}
\begin{proof} We will show that, for an element $s\in R_n$ with $\rk(s)=k$, computing all 
\begin{align*}
& \zeta_f(s,\{\},\{\}) = \zeta_f(s), \\
& \zeta_f(s,\{d_1(s)\},\{r_1(s)\}), \\
& \zeta_f(s,\{d_1(s),d_2(s)\},\{r_1(s),r_2(s)\}), \\
& \vdots \\
& \zeta_f(s,\{d_1(s),d_2(s),\ldots,d_{n-k}(s)\},\{r_1(s),r_2(s),\ldots,r_{n-k}(s)\}) 
\end{align*}
requires at most
\[
(n-k)^2 + \frac{(n-k-1)(n-k)(2n-2k-1)}{6}
\]
additions, assuming that steps $0,1,\ldots,n-k-1$ have already been completed. 

Let $s*(d_i(s)\rightarrow r_j(s))$ denote the element of $R_n$ that is obtained by adding $d_i(s)$ to the domain of $s$ and sending it to $r_j(s)$. For example, if 
\[
s = \left(\begin{array}{ccccccc}
1 &2 &3 &4 &5 &6 &7 \\
2 &- &5 &- &- &- &3 \end{array} 
\right),
\]
then
\[
s*(d_2(s)\rightarrow r_3(s))=
\left(\begin{array}{ccccccc}
1 &2 &3 &4 &5 &6 &7 \\
2 &- &5 &6 &- &- &3 \end{array} 
\right).
\]

Now, consider the poset of elements $t\in R_n$ with $t\geq s$. This poset is isomorphic to the poset for $R_{n-k}$, with $s$ in the place of the $0$ element. Following the idea in the proof of Theorem \ref{ThmSizeRnRecursiveTwo}, we have:
\begin{align*}
\zeta_f(s,\{\},\{\}) &= \zeta_f(s*(d_1(s)\rightarrow r_1(s)),\{\},\{\}) \\
 &+ \zeta_f(s*(d_2(s)\rightarrow r_1(s)),\{\},\{\}) + \cdots \\
 &+ \zeta_f(s*(d_{n-k}(s)\rightarrow r_1(s)),\{\},\{\}) \\
 &+ \zeta_f(s*(d_1(s)\rightarrow r_2(s)),\{\},\{\}) + \cdots \\
 &+ \zeta_f(s*(d_1(s)\rightarrow r_{n-k}(s)),\{\},\{\}) \\
 &- \sum_{i,j\in \{2,\ldots,n-k\}} \zeta_f(s*(d_i(s)\rightarrow r_1(s))*(d_1(s)\rightarrow r_j(s)),\{\},\{\}) \\
 &+ \zeta_f(s,\{d_1(s)\},\{r_1(s)\}).
\end{align*}
Notice that every term in this sum, with the exception of $\zeta_f(s,\{d_1(s)\},\{r_1(s)\})$, was computed in an earlier step. After all, $\rk(s*(d_i(s)\rightarrow r_j(s))) > \rk(s)$ for all $i,j$. Thus, once we have $\zeta_f(s,\{d_1(s)\},\{r_1(s)\})$, all we have to do to compute $\zeta_f(s)$ is add these terms up. To compute $\zeta_f(s,\{d_1(s)\},\{r_1(s)\})$, we again follow the idea in the proof of Theorem \ref{ThmSizeRnRecursiveTwo}, and we write:
\begin{align*}
\zeta_f(s,\{d_1(s)\},\{r_1(s)\}) &= \zeta_f(s*(d_2(s)\rightarrow r_2(s)),\{d_1(s)\},\{r_1(s)\}) \\
 &+ \zeta_f(s*(d_3(s)\rightarrow r_2(s)),\{d_1(s)\},\{r_1(s)\}) + \cdots \\
 &+ \zeta_f(s*(d_{n-k}(s)\rightarrow r_2(s)),\{d_1(s)\},\{r_1(s)\}) \\
 &+ \zeta_f(s*(d_2(s)\rightarrow r_3(s)),\{d_1(s)\},\{r_1(s)\}) + \cdots \\
 &+ \zeta_f(s*(d_2(s)\rightarrow r_{n-k}(s)),\{d_1(s)\},\{r_1(s)\}) \\
 &- \sum_{i,j\in \{3,\ldots,n-k\}} \zeta_f(s*(d_i(s)\rightarrow r_2(s))*\\
 &\quad \quad\quad (d_2(s)\rightarrow r_j(s)),\{d_1(s)\},\{r_1(s)\}) \\
 &+ \zeta_f(s,\{d_1(s),d_2(s)\},\{r_1(s),r_2(s)\}).
\end{align*}
Notice that $i,j \geq 2$ implies
\begin{align*}
& \zeta_f(s*(d_i(s)\rightarrow r_j(s)),\{d_1(s)\},\{r_1(s)\}) \\ &= \zeta_f(s*(d_i(s)\rightarrow r_j(s)),\{d_1(s*(d_i(s)\rightarrow r_j(s)))\},\{r_1(s*(d_i(s)\rightarrow r_j(s)))\}),
\end{align*}
so every term in this sum, with the exception of $\zeta_f(s,\{d_1(s),d_2(s)\},\{r_1(s),r_2(s)\})$, was computed in an earlier step.

In general, suppose $D= \{d_1(s),d_2(s),\ldots,d_m(s)\}$, $R=\{r_1(s),r_2(s),\ldots,r_m(s)\}$. If $m=n-k$, then we have
\[
\zeta_f(s,D,R)=f(s).
\]
If $m=n-k-1$, then we have
\begin{align*}
\zeta_f(s,D,R)&=\zeta_f(s*(d_{n-k}(s)\rightarrow r_{n-k}(s)),D,R) \\&+ \zeta_f(s,D \cup \{d_{n-k}(s)\},R \cup \{r_{n-k}(s)\}).
\end{align*}
Otherwise, $m<n-k-1$, and we have
\begin{align}
\zeta_f(s,D,R) &= \zeta_f(s*(d_{m+1}(s)\rightarrow r_{m+1}(s)),D,R) \notag\\
 &+ \zeta_f(s*(d_{m+2}(s)\rightarrow r_{m+1}(s)),D,R) + \cdots \notag\\
 &+ \zeta_f(s*(d_{n-k}(s)\rightarrow r_{m+1}(s)),D,R) \notag\\
 &+ \zeta_f(s*(d_{m+1}(s)\rightarrow r_{m+2}(s)),D,R) + \cdots \label{TwoLevels}\\
 &+ \zeta_f(s*(d_{m+1}(s)\rightarrow r_{n-k}(s)),D,R) \notag\\
 &- \sum_{i,j\in \{{m+2},\ldots,n-k\}} \zeta_f(s*(d_i(s)\rightarrow r_{m+1}(s))*\notag\\
 & \quad \quad\quad (d_{m+1}(s)\rightarrow r_j(s)),D,R) \notag\\
 &+ \zeta_f(s,D\cup \{d_{m+1}(s)\},R\cup\{r_{m+1}(s)\}), \notag
\end{align}
where every term in the sum, with the exception of 
\[
\zeta_f(s,D\cup \{d_{m+1}(s)\},R\cup\{r_{m+1}(s)\}),
\]
was computed in an earlier step of the algorithm. Once we have computed $\zeta_f(s,D\cup \{d_{m+1}(s)\},R\cup\{r_{m+1}(s)\})$, the number of operations required to compute $\zeta_f(s,D,R)$ is thus no more than
\[
(n-k-m)+(n-k-m-1)+(n-k-m-1)^2.
\]
We do this for $m$ from $n-k$ to $0$ to compute, in order:
\begin{align*}
& \zeta_f(s,\{d_1(s),d_2(s),\ldots,d_{n-k}(s)\},\{r_1(s),r_2(s),\ldots,r_{n-k}(s)\}), \\
& \vdots \\
& \zeta_f(s,\{d_1(s),d_2(s)\},\{r_1(s),r_2(s)\}), \\
& \zeta_f(s,\{d_1(s)\},\{r_1(s)\}), \\
& \zeta_f(s,\{\},\{\}) = \zeta_f(s),
\end{align*}
which is what we want. The total number of operations required is thus no more than
\begin{align*}
&\sum_{m=0}^{n-k} (n-k-m)+(n-k-m-1)+(n-k-m-1)^2 \\&= (n-k)^2 + \frac{(n-k-1)(n-k)(2n-2k-1)}{6}.
\end{align*}
\end{proof}

This algorithm yields the bound
\begin{thm}For $n\geq 3$,
$
\Cplx (\zeta_{R_n}) \leq \frac{2}{3}n^3|R_n|.
$
\end{thm}
\begin{proof}
Using the algorithm given above, we have
\[
\Cplx(\zeta_{R_n}) \leq \sum_{k=0}^n \left((n-k)^2 + \frac{(n-k-1)(n-k)(2n-2k-1)}{6} \right) \binom{n}{k}^2 k!
\]
When $n\geq 3$,
\begin{align*}
(n-k)^2 + \frac{(n-k-1)(n-k)(2n-2k-1)}{6} &\leq n^2 + \frac{2n^3}{6}\\
& \leq \frac{2}{3}n^3,
\end{align*}
in which case we have
\[
\Cplx (\zeta_{R_n}) \leq \frac{2}{3}n^3 \sum_{k=0}^n \binom{n}{k}^2 k! = \frac{2}{3}n^3|R_n|.
\]
\end{proof}

Combining this fast zeta transform with Theorem \ref{ThmRookPart1}, we obtain 
\begin{thm} $\Cplx(R_n) = O(|R_n| \log^3 |R_n|)$.
\label{ThmCplxSemigroupRn}
\end{thm}
\begin{proof}For $n\geq 3$, we have
\[
\Cplx(R_n) \leq \Cplx(\zeta_{R_n}) + \frac{3}{4}n(n-1)|R_n|
\leq \frac{2}{3}n^3|R_n| + \frac{3}{4}n(n-1)|R_n|.
\]
Since $|R_n|\geq n!$ and $n = O(\log (n!))$, we are done.
\end{proof}

\noindent{\bf Remark:} This result considerably improves the bound on $\Cplx(R_n)$ we obtained in \cite[Theorem 8.2]{RookFFT}, where we used a naive implementation of the zeta transform on $R_n$ to prove  
\[
\Cplx(R_n) \leq 2^n|R_n| + \frac{3}{4}n(n-1)|R_n|.
\]

%


We now turn to an analysis of the memory required to use the fast zeta transform. Since we compute $n-k+1$ complex numbers for each element of rank $k$, it is immediate that we need to store no more than $(n+1)|R_n|$ complex numbers during runtime. However, by (\ref{TwoLevels}), we only need partial zeta transforms of elements of rank $k+1$ and $k+2$ to compute the partial zeta transforms for an element of rank $k$, so when we begin step $n-k$ of our algorithm (compute all partial zeta transforms for the elements of rank $k$) we may discard the partial zeta transforms $\zeta_f(s,A,B)$ for $A,B\neq \{\}$ and $\rk(s)>k+2$.

Storing the inputs (the $f(s)$) and the outputs (the $\zeta_f(s)$) of the zeta transform requires the storage of $2|R_n|$ complex numbers. At the beginning of step $n-k$, let us allocate memory for all of the partial zeta transforms of the elements of rank $k$ and discard the partial zeta transforms of rank $k+3$ and higher that we no longer need. Note that for any $s\in R_n$, one of the $\zeta_f(s,A,B)$ that we compute is just $f(s)$ and another is $\zeta_f(s)$, so at the beginning of step $n-k$ we allocate memory for $(n-k-1)\binom{n}{k}^2k!$ complex numbers that will be discarded later (when $k=n$ we interpret this to be 0).

\begin{thm}This fast zeta transform requires the storage of no more than $O(n^{1/4}|R_n|)$ complex numbers in memory at any point during its execution.
\label{ThmZetaMemoryUpper}
\end{thm}

\begin{proof} At the beginning of step $n-k$ we allocate memory for $(n-k-1)\binom{n}{k}^2k!$ complex numbers that will be discarded later. We continue to store all of the partial zeta transforms of the elements only of rank $k+1$ and $k+2$ as well, so at no point will we ever need to store more than
\[
2|R_n| +  3\left(\max_{k\in\{0,1,\ldots,n-1\}} (n-k-1)\binom{n}{k}^2k!\right)
\]
complex numbers in memory during the execution of the fast zeta transform. We claim
\[
\max_{k\in\{0,1,\ldots,n-1\}} (n-k-1)\binom{n}{k}^2k! = O(n^{1/4}|R_n|).
\]
To see this, write $k=n-x\sqrt{n}$ for some $0< x \leq \sqrt{n}$. Then using Stirling's approximation we have
\begin{align}
(n-k-1)\binom{n}{k}^2k! &\leq \frac{(n-k)n!n!}{k!(n-k)!(n-k)!} \notag\\ 
&=\frac{x\sqrt{n}n!n!}{(n-x\sqrt{n})!(x\sqrt{n})!^2} \notag\\
&\leq \frac{n!n!}{(n-x\sqrt n)!2\pi ({x\sqrt n}/{e})^{2x\sqrt n}e^{2/(12x\sqrt{n}+1)}} \notag\\ 
&=\frac{n^{1/4}\sqrt{e}}{\sqrt{\pi}}\cdot 
\frac{n!e^{2\sqrt{n}}}{2n^{1/4}\sqrt{\pi e}} \cdot  
\frac{n!}{(n-x\sqrt{n})!\sqrt{n}^{2x\sqrt{n}}}\cdot 
\frac{e^{2x\sqrt{n}}}{x^{2x\sqrt{n}}e^{2\sqrt{n}}}\cdot 
\frac{1}{e^{2/(12x\sqrt{n}+1)}}.\label{ugh1}
\end{align}

By \cite[Theorem 1]{Asymptotics}, we have that $|R_n|$ is asymptotically $n!e^{2\sqrt{n}}/(2n^{1/4}\sqrt{\pi e})$, so for $\epsilon>0$ let $n$ be large enough so that $n!e^{2\sqrt{n}}/(2n^{1/4}\sqrt{\pi e})\leq (1+\epsilon)|R_n|$. All of the terms in the product (\ref{ugh1}) are nonnegative, and we show below that the final three terms are each bounded above by 1. From this it follows that for $n$ large enough, for all $k$ we have
\[
(n-k-1)\binom{n}{k}^2k! \leq n^{1/4} \frac{\sqrt{e}}{\sqrt{\pi}} (1+\epsilon)|R_n|,
\]
so for $n$ large enough, for all $k$ we have
\[
(n-k-1)\binom{n}{k}^2k! \leq n^{1/4}|R_n|,
\]
and thus for $n$ large enough the number of complex numbers we need to store in memory during the execution of the fast zeta transform is bounded by  $2|R_n|+3n^{1/4}|R_n|$. 

To finish the proof, then, we show that the final three terms of the product (\ref{ugh1}) are bounded above by 1. First,
\begin{align*}
\log\left(\frac{n!}{(n-x\sqrt{n})!\sqrt{n}^{2x\sqrt{n}}}\right) &= \log(n)+\log(n-1)+\cdots+\log(n-x\sqrt{n}+1) - x\sqrt{n}\log(n)\\
&\leq x\sqrt{n}\log(n) - x\sqrt{n}\log(n) = 0,
\end{align*}
so
\[
\frac{n!}{(n-x\sqrt{n})!\sqrt{n}^{2x\sqrt{n}}} \leq 1.
\]
Next, for fixed $n$ and $x\in(0,\sqrt{n}]$, elementary calculus shows that
$$
f(x)=\frac{e^{2x\sqrt{n}}}{x^{2x\sqrt{n}}}
$$
is increasing for $x\in(0,1]$ and decreasing for $x\in[1,\sqrt{n}]$. Hence $f(x)$ is maximized at $x=1$, where it has value $e^{2\sqrt{n}}$. Thus
\[
\frac{e^{2x\sqrt{n}}}{x^{2x\sqrt{n}}e^{2\sqrt{n}}} \leq 1.
\]
Finally, it is immediate that
\[
\frac{1}{e^{2/(12x\sqrt{n}+1)}}\leq 1,
\]
which completes the proof.
\end{proof}

Next, we show that the bound given in Theorem \ref{ThmZetaMemoryUpper} is, up to $O$, the best possible.
\begin{thm}There exist infinitely many $n$ for which this fast zeta transform requires the storage of at least $\frac{1}{3}n^{1/4}|R_n|$ complex numbers in memory at some point during its execution.\label{ThmZetaMemoryLower}
\end{thm}
\begin{proof}Let $n$ be a square. Proceeding rank by rank, at some point we will have all of the partial zeta transforms of all of the elements of rank $n-\sqrt{n}$ in memory. At this point, we will be storing
\[
(\sqrt{n}-1)\binom{n}{n-\sqrt{n}}^2(n-\sqrt{n})!
\]
complex numbers which we intend to discard later.
Now, using Stirling's approximation we have
\begin{align}
(\sqrt{n}-1)\binom{n}{n-\sqrt{n}}^2(n-\sqrt{n})! &= \frac{(\sqrt{n}-1)n!n!}{(n-\sqrt{n})!(\sqrt{n}!)^2} \notag\\
&\geq \frac{(\sqrt{n}-1)n!n!}{(n-\sqrt{n})! 2\pi \sqrt{n} (\sqrt{n}/e)^{2\sqrt{n}}e^{1/(6\sqrt{n})}} \notag\\
&= \frac{n^{1/4}\sqrt{e}}{\sqrt{\pi}}\cdot 
\frac{n!e^{2\sqrt{n}}}{2n^{1/4}\sqrt{\pi e}} \cdot 
\frac{n!}{(n-\sqrt{n})!n^{\sqrt{n}}} \cdot
\frac{\sqrt{n}-1}{\sqrt{n}}\cdot
\frac{1}{e^{1/(6\sqrt{n})}}. \label{ugh2}
\end{align}

By \cite[Theorem 1]{Asymptotics}, we have that $|R_n|$ is asymptotically $n!e^{2\sqrt{n}}/(2n^{1/4}\sqrt{\pi e})$, so for $\epsilon>0$ let $n$ be large enough so that: 
\begin{itemize}
	\item $n!e^{2\sqrt{n}}/(2n^{1/4}\sqrt{\pi e})\geq (1-\epsilon)|R_n|$,
	\item $(1-1/\sqrt{n})^{\sqrt{n}}\geq (1-\epsilon)\cdot1/e$,
	\item $(\sqrt{n}-1)/{\sqrt{n}} \geq 1-\epsilon$, and
	\item ${1}/(e^{1/(6\sqrt{n})}) \geq 1-\epsilon$.
\end{itemize}
Note that
\begin{align*}
\log\left(\frac{n!}{(n-\sqrt{n})!n^{\sqrt{n}}}\right) &= \log(n)+\log(n-1)+\cdots+\log(n-\sqrt{n}+1)-\sqrt{n}\log(n) \\
&\geq \sqrt{n}\log(n-\sqrt{n})-\sqrt{n}\log(n) = \log\left(\left(1-\frac{1}{\sqrt{n}}\right)^{\sqrt{n}}\right),
\end{align*}
so
\[
\frac{n!}{(n-\sqrt{n})!n^{\sqrt{n}}} \geq \left(1-\frac{1}{\sqrt{n}}\right)^{\sqrt{n}}.
\]
Combining this with the fact that all of the terms in the product (\ref{ugh2}) are nonnegative, we obtain
\begin{align*}
(\sqrt{n}-1)\binom{n}{n-\sqrt{n}}^2(n-\sqrt{n})! &\geq 
\frac{n^{1/4}}{\sqrt{\pi e}}\cdot 
(1-\epsilon)^4|R_n|,
\end{align*}
so for all squares $n$ large enough we have
\[
(\sqrt{n}-1)\binom{n}{n-\sqrt{n}}^2(n-\sqrt{n})! \geq \frac{1}{3}n^{1/4}|R_n|.
\]
\end{proof}

\section{FFTs for Rook Wreath Products}
\label{SecWreath}

Let $G$ be a finite group. The {\em rook wreath product} $\Gwr{R_n}$ is the semigroup of all $n\times n$ matrices with entries in $\{0\} \cup G$ having at most one non-zero entry per row and column under the operation of matrix multiplication. Let us write $1$ for the identity of $G$. Clearly, then, we recover the rook monoid $R_n$ as $\Z_1 \wr R_n$. It is easy to see that the idempotents of $\Gwr R_n$ are precisely the idempotents of $R_n$, and that $\Gwr R_n$ is an inverse semigroup.

In this section we again use the approach detailed in Section \ref{secGeneralApproach} to create an $O(|\Gwr R_n|\log^4 |\Gwr R_n|)$-complexity FFT for $\Gwr R_n$. We begin by extending notions about the rook monoid to $\Gwr R_n$ and recording a number of facts about $\Gwr R_n$.


\subsection{Facts about $\Gwr R_n$}
\label{SecWreathInfo}

\begin{defn}For an element $s\in \Gwr R_n$, the {\em rank} of s, denoted $\rk(s)$, is the number of rows of $s$ which contain nonzero entries. \end{defn}
Equivalently, $\rk(s)$ is the number of columns of $s$ which contain nonzero entries.

\begin{defn}The {\em symmetric group wreath product} $\Gwr S_n$ is the group of all $n\times n$ matrices with entries in $\{0\} \cup G$ having exactly one non-zero entry per row and column. The operation on $\Gwr S_n$ is matrix multiplication.\end{defn}
Thus $\Gwr S_n$ is contained in $\Gwr R_n$ as the rank-$n$ elements.

We generalize the notion of domain and range from $R_n$ to $\Gwr R_n$ as follows.
\begin{defn}
Let $s\in \Gwr R_n$. Define $\dom(s)$ to be the set of indices of the columns of $s$ which contain nonzero entries, and $\ran(s)$ to be the set of indices of the rows of $s$ which contain nonzero entries. 
\end{defn}
This definition agrees with our previous definitions of inverse semigroup domain and range, that is,
\[
\dom(s) = s^{-1}s, \quad \ran(s) = ss^{-1},
\]
provided that we once again abuse the distinction between the domain and range of a map and the corresponding partial identities (as elements of $R_n$).

We must understand the maximal subgroups of $\Gwr R_n$. Let $e\in \Gwr R_n$ be idempotent, with $\rk(e)=k$.
\begin{thm}The maximal subgroup of $\Gwr R_n$ at $e$ is isomorphic to $\Gwr S_k$.\end{thm}
\begin{proof}Denote this subgroup by $G_e$. We have
\[
G_e = \{s\in \Gwr R_n: ss^{-1} = s^{-1}s = e\}.
\]
Suppose that $\dom(e)=\{i_1,\ldots,i_k\}$ (and hence $\ran(e)=\{i_1,\ldots,i_k\}$, because $e$ is idempotent).
Clearly, then,
\[
\{x \in \Gwr R_n : \dom(x)=\ran(x)=\{i_1,\ldots,i_k\}\} \subseteq G_e.
\]
If $x\in G_e$, then $\dom(x)=\ran(x)$. Furthermore, if $\rk(x)\neq k$, then $x\notin G_e$. If $j\in \dom(x)$ with $j\notin \{i_1,\ldots,i_k\}$, then $x^{-1}x\neq e$, and so $x\notin G_e$. Thus
\[
G_e = \{x \in \Gwr R_n : \dom(x)=\ran(x)=\{i_1,\ldots,i_k\}\},
\]
which is isomorphic to $\Gwr S_k$ in the obvious way (i.e., for $x\in G_e$, delete the rows and columns of $x$ which contain only zeroes).
\end{proof}

We will also need to understand the poset structure of $\Gwr R_n$.
\begin{thm}Let $s,t\in \Gwr R_n$. Then $s\leq t$ if and only if $s$ may be obtained by replacing entries in $t$ with $0$.
\end{thm}
\begin{proof}This follows directly from the definition
\[
s\leq t \iff s=et \up{ for some idempotent }e\in \Gwr R_n
\]
together with the fact that the idempotents of $\Gwr R_n$ are the idempotents of $R_n$ (i.e., the restrictions of the identity matrix).
\end{proof}  

Finally, we record the size of $\Gwr R_n$.
\begin{thm}
\[
|\Gwr R_n| = \sum_{k=0}^n \binom{n}{k}^2 k! |G|^k.
\]
\end{thm}
\begin{proof}
There are $\binom{n}{k}^2 k!$ rook matrices of rank $k$, and for a given rook matrix $X$ of rank $k$, there are $|G|$ options to replace each of the $1$'s in $X$ with elements of $G$.
\end{proof}


\subsection{The FFT for $\Gwr R_n$}
\label{SecWreathFFT}

We now explain our FFT for $\Gwr R_n$. We begin by handling the term in Theorem \ref{MyBigThm} concerning the change of basis from the groupoid basis of $\Gwr R_n$ to a Fourier basis. We write $\C [\Gwr R_n]$ and $\C [\Gwr S_k]$ for the complex algebras of rook monoid and symmetric group wreath products. Theorem \ref{BigThm} and the discussion in Section \ref{SecWreathInfo} imply that we have
\[
\C [\Gwr R_n] \cong \bigoplus_{k=0}^n M_{\binom{n}{k}}(\C [\Gwr S_k]).
\]
Theorem \ref{MyBigThm} then applies and yields
\[
\Cplx(\Gwr R_n) \leq \Cplx(\zeta_{\Gwr R_n}) + \sum_{k=0}^n \binom{n}{k}^2 \Cplx(\Gwr S_k).
\]

In \cite{DanWreath}, D. Rockmore constructs a complete set of inequivalent, irreducible representations $\Y_k$ for $\C [\Gwr S_k]$ and proves the following \cite[Corollary 2]{DanWreath}.
\begin{thm}Let $h$ denote the number of inequivalent, irreducible representations of $G$. Then
\[
\T_{\Y_k}(\Gwr S_k)\leq k!|G|^k \cdot
\left[
\frac{\Cplx(G)}{|G|}\cdot\frac{k(k+1)}{2} + 2^h\frac{k^2(k+1)^2}{4} + 1
\right].
\]
\end{thm}
In particular, $\Cplx(\Gwr S_k)$ is bounded by the same amount. We take the $\Y_k$ and tensor them up to be our set of representations $\Y$ for $\C [\Gwr R_n]$.

\begin{thm}
\label{ThmCplxGroupoidGwrRn}We have
\[
\Cplx(\Gwr R_n) = \Cplx(\zeta_{\Gwr R_n}) + O(|\Gwr R_n| \log^4 |\Gwr R_n|).
\]
\end{thm}
\begin{proof}We have
\begin{align*}
\T_\Y(\Gwr R_n) 
&\leq \Cplx(\zeta_{\Gwr R_n}) + \sum_{k=0}^n \binom{n}{k}^2 k!|G|^k \cdot\left[\frac{\Cplx(G)}{|G|}\cdot\frac{k(k+1)}{2} + 2^h\frac{k^2(k+1)^2}{4} + 1 \right] \\
&\leq \Cplx(\zeta_{\Gwr R_n}) + \left[\frac{\Cplx(G)}{|G|}\cdot\frac{n(n+1)}{2} + 2^h\frac{n^2(n+1)^2}{4} + 1 \right] \cdot \sum_{k=0}^n \binom{n}{k}^2 k!|G|^k \\
&\leq \Cplx(\zeta_{\Gwr R_n}) + \left[\frac{\Cplx(G)}{|G|}\cdot\frac{n(n+1)}{2} + 2^h\frac{n^2(n+1)^2}{4} + 1 \right] \cdot |\Gwr R_n|.
\end{align*}
Now, $|G|$, $\Cplx(G)$, and $2^h$ are constants with respect to $n$, and $n=O(\log |\Gwr R_n|)$. The theorem follows.
\end{proof}

Now, let $f\in \C [\Gwr R_n]$ be an arbitrary element, expressed with respect to the semigroup basis:
\[
f= \sum_{s\in \Gwr R_n}f(s)s.
\]
We would like to express $f$ with respect to the groupoid basis:
\[
f=\sum_{s\in \Gwr R_n}g(s)\s,
\]
where, by (\ref{sissumofbrackett}), the coefficients $g(s)$ are given by
\[
g(s)=\sum_{\substack{t\in \Gwr R_n: \\t\geq s}} f(t).
\]
Our goal is to compute the coefficients $g(s)$ in an efficient manner, and we give an algorithm below for doing so. Note that the algorithm below reduces to the algorithm for the fast zeta transform for $R_n$ given in Section \ref{SecRookFFT} when $G=\Z_1$. As in Section \ref{SecRookFFT}, we begin by proving a recursive formula for the size of $\Gwr R_n$.

\begin{thm}
\label{ThmSizeGwrRnRecursive}
For $n\geq 3$,
\[
|\Gwr R_n| = (2n-1)|G||\Gwr R_{n-1}| + |\Gwr R_{n-1}| - (n-1)^2|G|^2|\Gwr R_{n-2}|.
\]
\end{thm}
\begin{proof}
$\Gwr R_n$ consists of those elements having all 0's in column $1$ and row $1$ (of which there are $|\Gwr R_{n-1}|$), together with, for each $x\in G$ and $\alpha \in \{1,\ldots, n\}$, those having an $x$ in position $(\alpha, 1)$ (of which there are $n|G||\Gwr R_{n-1}|$ total), together with, for each $x\in G$ and $\alpha \in \{2,\ldots, n\}$, those having an $x$ in position $(1,\alpha)$ (of which there are $(n-1)|G||\Gwr R_{n-1}|$ total). Counting the number of elements of $\Gwr R_n$ in this way overcounts. For each pair $\alpha,\beta$ with $2 \leq \alpha , \beta \leq n$ and for each pair of elements $x,y\in G$, every element with $x$ position $(\alpha,1)$ and $y$ in position $(1,\beta)$ (of which there are $(n-1)^2|G|^2|\Gwr R_{n-2}|$ total) gets counted twice.
\end{proof}

We now explain the fast zeta transform. As in Section \ref{SecRookFFT}, let us denote $\sum_{t\geq s}f(t)$ by $\zeta_f(s)$. Let $s\in \Gwr R_n$. The rows and columns of $s$ are indexed by $\{1,2,\ldots,n\}$.
\begin{itemize}
\item Let $d_i(s)$ be the index of the $i^{th}$ column of $s$ which contains only zeroes.
\item Let $r_i(s)$ be the index of the $i^{th}$ row of $s$ which contains only zeroes.
\end{itemize}
Define ``partial'' zeta transforms at $s$ as follows:
\begin{align*}
\zeta_f(s,\{d_1(s),d_2(s),&\ldots,d_{m}(s)\},\{r_1(s),r_2(s),\ldots,r_{m}(s)\}) = \\ & \sum_{\substack{t\geq s: \\ \up{columns }d_1(s),\ldots,d_m(s)\up{ of }t \up{ contain only zeroes and}\\\up{rows }r_1(s),\ldots,r_m(s)\up{ of }t \up{ contain only zeroes} }}f(t).
\end{align*}

As with $R_n$, we work from the ``top'' down, and our zeta transform proceeds as follows, with steps $0,1,\ldots, n$:
\begin{itemize}
\item Step 0: For all $s\in \Gwr R_n$ with $\rk(s)=n$, compute all $\zeta_f(s,\{\},\{\})=\zeta_f(s)$  (0 operations).
\item Step 1: For all $s\in \Gwr R_n$ with $\rk(s)=n-1$, compute $\zeta_f(s,\{\},\{\})=\zeta_f(s)$ and $\zeta_f(s,\{d_1(s),r_1(s)\})$ ($|G|$ operations for each element $s$).
\[
\vdots
\]
\item Step $n-k$: For all $s\in \Gwr R_n$ with $\rk(s)=k$, compute all
\begin{align*}
& \zeta_f(s,\{\},\{\}) = \zeta_f(s), \\
& \zeta_f(s,\{d_1(s)\},\{r_1(s)\}), \\
& \zeta_f(s,\{d_1(s),d_2(s)\},\{r_1(s),r_2(s)\}), \\
& \vdots \\
& \zeta_f(s,\{d_1(s),d_2(s),\ldots,d_{n-k}(s)\},\{r_1(s),r_2(s),\ldots,r_{n-k}(s)\}).
\end{align*}
\[
\vdots
\]
\end{itemize}

Thus, instead of computing just $\zeta_f(s)$ for the elements $s$ of rank $k$, we compute $\zeta_f(s)$ along with $n-k$ other numbers for each element $s$ of rank $k$. These other numbers are needed for the efficient computation of the zeta transform at elements of rank $k-1$ and $k-2$, and can be discarded when they are no longer needed. We are currently unable to give precise bounds on the amount of memory required for this algorithm, partly due to the lack of an asymptotic formula for \mbox{$|\Gwr R_n|$}. However, it is immediate that it requires the storage of no more than \mbox{$(n+1)|\Gwr R_n|$} complex numbers in memory during runtime.

\begin{thm}Step $n-k$ requires at most
\[
\left(|G|(n-k)^2 + |G|^2 \frac{(n-k-1)(n-k)(2n-2k-1)}{6}\right)
\binom{n}{k}^2 k! |G|^k
\]
operations in total.
\end{thm}

\begin{proof} We will show that, for an element $s\in \Gwr R_n$ with $\rk(s)=k$, computing all 
\begin{align*}
& \zeta_f(s,\{\},\{\}) = \zeta_f(s), \\
& \zeta_f(s,\{d_1(s)\},\{r_1(s)\}), \\
& \zeta_f(s,\{d_1(s),d_2(s)\},\{r_1(s),r_2(s)\}), \\
& \vdots \\
& \zeta_f(s,\{d_1(s),d_2(s),\ldots,d_{n-k}(s)\},\{r_1(s),r_2(s),\ldots,r_{n-k}(s)\}) 
\end{align*}
requires at most
\[
|G|(n-k)^2 + |G|^2 \frac{(n-k-1)(n-k)(2n-2k-1)}{6}
\]
additions, assuming that steps $0,1,\ldots,n-k-1$ have already been completed. 

Suppose $G=\{g_1,g_2,\ldots,g_{|G|}\}$. Let $s*(g_y E_{r_j(s),d_i(s)})$ denote the element of $\Gwr R_n$ that is obtained by inserting $g_y$ into the $r_j(s),d_i(s)$ position of $s$. For example, if 
\[
s = \left[\begin{array}{ccccccc}
0&0&0&0&0&0&0\\ 
g_{y_1}&0&0&0&0&0&0\\
0&0&0&0&0&0&g_{y_3}\\
0&0&0&0&0&0&0\\
0&0&g_{y_2}&0&0&0&0\\
0&0&0&0&0&0&0\\
0&0&0&0&0&0&0
\end{array}\right]
\]
then
\begin{center}
\begin{tabular}{ll}
$d_1(s) = 2$,  &$r_1(s)= 1$, \\ 
$d_2(s) = 4$,  &$r_2(s)= 4$, \\ 
$d_3(s) = 5$,  &$r_3(s)= 6$, \\
$d_4(s) = 6$,  &$r_4(s)= 7$, 
\end{tabular}
\end{center}
and
\[s*(g_{y_4}E_{r_3(s),d_2(s)})=
\left[\begin{array}{ccccccc}
0&0&0&0&0&0&0\\ 
g_{y_1}&0&0&0&0&0&0\\
0&0&0&0&0&0&g_{y_3}\\
0&0&0&0&0&0&0\\
0&0&g_{y_2}&0&0&0&0\\
0&0&0&g_{y_4}&0&0&0\\
0&0&0&0&0&0&0
\end{array}\right].
\]

Now, consider the poset of elements $t\in \Gwr R_n$ with $t\geq s$. This poset is isomorphic to the poset for $\Gwr R_{n-k}$, with $s$ in the place of the $0$ matrix. 

As in Section \ref{SecRookFFT}, we compute our $n-k$ partial zeta transforms at $s$ in the following order. First, we have
\[
\zeta_f(s,\{d_1(s),d_2(s),\ldots,d_{n-k}(s)\},\{r_1(s),r_2(s),\ldots,r_{n-k}(s)\}) = f(s),
\]
which requires no operations.
Next, let $D = \{d_1(s),d_2(s),\ldots,d_{n-k-1}(s)\}$ and $R=\{r_1(s),r_2(s),\ldots,r_{n-k-1}(s)\}$. We have
\begin{align*}
\zeta_f(s,D,R) &= \sum_{i=1}^{|G|}\zeta_f(s*(g_iE_{r_{n-k}(s),d_{n-k}(s)}),D,R) \\
&+\zeta_f(s,\{d_1(s),d_2(s),\ldots,d_{n-k}(s)\},\{r_1(s),r_2(s),\ldots,r_{n-k}(s)\}),
\end{align*}
which requires $|G|$ operations.

Now, suppose $D = \{d_1(s),d_2(s),\ldots,d_{m}(s)\}$ and $R=\{r_1(s),r_2(s),\ldots,r_{m}(s)\}$, with $m<n-k-1$. Following the proof of Theorem \ref{ThmSizeGwrRnRecursive}, we have
\begin{align*}
\zeta_f(s,D,R)&=
  \sum_{i=1}^{|G|}\zeta_f(s*(g_iE_{r_{m+1}(s),d_{m+1}(s)}),D,R) \\
&+\sum_{i=1}^{|G|}\zeta_f(s*(g_iE_{r_{m+1}(s),d_{m+2}(s)}),D,R) \\
&+\cdots \\ 
&+\sum_{i=1}^{|G|}\zeta_f(s*(g_iE_{r_{m+1}(s),d_{n-k}(s)}),D,R) \\
&+\sum_{i=1}^{|G|}\zeta_f(s*(g_iE_{r_{m+2}(s),d_{m+1}(s)}),D,R) \\
&+\cdots \\
&+\sum_{i=1}^{|G|}\zeta_f(s*(g_iE_{r_{n-k}(s),d_{m+1}(s)}),D,R) \\
&-\sum_{\substack{i,j\in \{m+2,\ldots,n-k\} \\ k,l \in \{1,2,\ldots,|G|\}}} \zeta_f(s*(g_kE_{r_{m+1}(s),d_{i}(s)})*(g_lE_{r_{j}(s),d_{m+1}(s)}),D,R) \\
&+\zeta_f(s,D\cup\{d_{m+1}(s)\},R\cup\{r_{m+1}(s)\}).
\end{align*}
Computing $\zeta_f(s,D,R)$ therefore requires no more than
\[
|G|(2n-2k-2m-1) + |G|^2(n-k-m-1)^2
\]
operations. We do this for $m=n-k$ to $0$ to compute, in order,
\begin{align*}
& \zeta_f(s,\{d_1(s),d_2(s),\ldots,d_{n-k}(s)\},\{r_1(s),r_2(s),\ldots,r_{n-k}(s)\}), \\
& \vdots \\
& \zeta_f(s,\{d_1(s),d_2(s)\},\{r_1(s),r_2(s)\}), \\
& \zeta_f(s,\{d_1(s)\},\{r_1(s)\}), \\
& \zeta_f(s,\{\},\{\}) = \zeta_f(s),
\end{align*}
which is what we want. The total number of operations required to compute these is thus no more than
\begin{align*}
&|G| + \sum_{m=0}^{n-k-2}|G|(2n-2k-2m-1) + |G|^2(n-k-m-1)^2 \\
&=|G|(n-k)^2 + |G|^2\frac{(n-k-1)(n-k)(2n-2k-1)}{6}.
\end{align*}
\end{proof}

This algorithm yields the following bound.
\begin{thm}
\label{ThmWreathZetaCplx}
$\Cplx(\zeta_{\Gwr R_n}) = O(|\Gwr R_n| \log^3 |\Gwr R_n|)$.
\end{thm}
\begin{proof}For $n\geq 3$,
\begin{align*}
|G|(n-k)^2 + |G|^2\frac{(n-k-1)(n-k)(2n-2k-1)}{6} & \leq |G|n^2 + |G|^2\frac{2n^3}{6} \\
&\leq \frac{2}{3}|G|^2 n^3.
\end{align*}
Hence
\[
\Cplx(\zeta_{\Gwr R_n}) \leq \frac{2}{3} |G|^2 n^3 |\Gwr R_n|.
\]
Since $|G|$ is a constant with respect to $n$ and $n=O(\log |\Gwr R_n|)$, we are done.
\end{proof}

Combining Theorems \ref{ThmCplxGroupoidGwrRn} and \ref{ThmWreathZetaCplx}, we obtain:
\begin{thm} $\Cplx(\Gwr R_n) = O(|\Gwr R_n| \log^4 |\Gwr R_n|)$.
\label{ThmCplxSemigroupGwrRn}
\end{thm}
\begin{proof}We have
\begin{align*}
\Cplx(\Gwr R_n) &= \Cplx(\zeta_{\Gwr R_n}) + O(|\Gwr R_n| \log^4 |\Gwr R_n|) \\
&= O(|\Gwr R_n| \log^3 |\Gwr R_n|) + O(|\Gwr R_n| \log^4 |\Gwr R_n|) \\
&= O(|\Gwr R_n| \log^4 |\Gwr R_n|).
\end{align*}
\end{proof}

\section{Future Directions}
\label{SecFutureDirections}

The generalization of the theory of Fourier transforms to inverse semigroups and
beyond presents a new set of interesting challenges. Theorem \ref{MyBigThm} opens the
door for the development of more FFTs on inverse semigroups, as it reduces the
problem of creating these Fourier transforms to the problems of creating
FFTs on their maximal subgroups and creating fast zeta transforms
on their poset structures. While the theory of group FFTs is well-developed, the
theory of fast zeta transforms is not. An interesting line of research, then, would
be to create a theory of fast zeta transforms for inverse semigroup posets. On the
other hand, the poset structure of an inverse semigroup can be about as bad as
one wants---any meet semilattice is possible. It remains to be seen whether there
are any guiding principles one might employ when creating fast zeta transforms.

We would also like to develop applications of these FFTs. In general, whereas groups capture global symmetries, inverse semigroups capture partial symmetries (see \cite{Lawson} for more on this idea). It is reasonable, then, that FFTs for certain inverse semigroups would be useful for data analysis in situations where FFTs on the analogous groups are useful. For example, the Fourier transform on the symmetric group has been used for the statistical analysis of voting data \cite{Persi}, and Fourier transforms on symmetric group wreath products may be used for the statistical analysis of nested designs \cite{DanWreath}. We are currently investigating applications of the Fourier transform on the rook monoid to the statistical analysis of voting datasets which contain incomplete voter preferences.



\bibliographystyle{plain}

\end{document}